\documentclass[12pt,a4paper]{amsart}
\usepackage{mathrsfs}
\usepackage{amsfonts}
\usepackage[active]{srcltx}
\usepackage{enumerate}
\usepackage[parfill]{parskip}

\usepackage{amsmath,amssymb,xspace,amsthm}

\newtheorem{thm}{Theorem}
\newtheorem{lem}[thm]{Lemma}

\newtheorem{cor}[thm]{Corollary}

\newtheorem{rem}[thm]{Remark}
\newtheorem{exa}[thm]{Example}
\numberwithin{equation}{section}

\def\Ind{\operatorname{Ind}}

\newcommand{\C}{\mathbb C}

\newcommand{\N}{\mathbb{N}}
\newcommand{\Z}{\mathbb{Z}}

\def\fa{\mathfrak{a}}
\renewcommand{\a}{\alpha}

\renewcommand{\phi}{\varphi}

\def\Aut{{\mathrm {Aut}}}

\def\supp{\mathrm{wt}}

\def\H{\mathcal{H}}

\def\fa{\mathfrak{age}(1)}

\def\S{\mathcal{S}}

\def\sl{\mathfrak{sl}}

\title[Simple weight modules over ageing algebra]{\bf Classification of simple weight modules over the
$1$-spatial ageing algebra}
\author{Rencai L\"u, Volodymyr Mazorchuk and Kaiming Zhao }

\date{}

\begin{document}

\begin{abstract}
In this paper we use Block's classification of simple modules over the first Weyl algebra to obtain a  complete
classification of simple weight modules, in particular, of Harish-Chandra modules, over the 1-spatial ageing algebra
$\fa$. Most of these modules have infinite dimensional weight spaces and so far the algebra $\fa$ is the only
Lie algebra having simple  weight modules with infinite dimensional weight spaces for which such a classification
exists. As an application we classify all simple weight modules over the $(1+1)$-dimensional
space-time Schr\"odinger algebra $\S$ that have a simple $\fa$-submodule thus constructing many new simple
weight $\S$-modules.
\end{abstract}
\vspace{5mm}
\maketitle

\begin{center}
\parbox{101mm}{{\bf Keywords:}  weight module, Schr\"odinger algebra,
ageing algebra, highest weight module}
\vspace{2mm}

\parbox{101mm}{\noindent{\bf Mathematics Subject Classification} 2010: 17B10, 17B30, 17B80}
\end{center}

\section{Introduction and description of the results}\label{s1}

Non-semisimple Lie algebras play an important role in physics where they are frequently used to study various
physical systems and explain diverse physical phenomena. For example one could mention Schr\"odinger algebras
and groups, see \cite{DDM}, conformal Galilei algebras and groups, see \cite{AI, AIK, CZ, GI1, GI2, HP, NOR, LMZ},
and ageing algebras, see \cite{HP, H1, H2, H3, H4, HEP, HU, HS1, HS2, PH, HSSU}. The representation theory of
finite dimensional semisimple Lie algebras is fairly well developed, see for example \cite{Hu,Ja,M} and the
references therein. In contrast to this, many important methods of the representation theory of semisimple algebra
are either not available or not yet developed or, at best, become much more complicated in the non-semisimple case,
confer e.g. \cite{DLMZ}. As a consequence, much less in known for non-semisimple case even for Lie algebras of
rather small dimension. Some recent results, see e.g. \cite{WZ,D,DLMZ,LMZ}, make some progress in studying modules
over certain non-semisimple extensions of the Lie algebra $\mathfrak{sl}_2$ motivated by their applications in physics.

The Schr\"odinger Lie group is the group of symmetries of the free particle Schr\"odinger equation.  The classical
{\em Schr\"odinger algebra} is the Lie algebra of this group. The main objects of study in this paper is the
extended Schr\"odinger algebra $\mathcal{S}$ in $(1+1)$-dimensional space-time and its ageing subalgebra $\fa$,
see precise definitions in Section~\ref{s2}. The name of the latter algebra comes from its use as dynamical
symmetry in physical ageing, which can be observed in strongly interacting many-body systems quenched from a
disordered initial state to the co-existence regime below the critical temperature $T_c > 0$ where several
equivalent equilibrium states exist, see \cite{BR} for details. Various representations of $\fa$ were constructed
in the literature, see for example \cite{HS1, HS2, H1, H2, H3, H4, HU, PH, HSSU}.

Both $\mathcal{S}$ and $\fa$ have natural Cartan subalgebras which allow to define the notion of weight modules,
that is modules on which elements of the  Cartan subalgebra act diagonalizably. The main objective of the present
paper is to classify all simple weight modules over $\fa$. It turns out that these modules have either all their weight spaces  one-dimensional, or  all their weight spaces infinite dimensional.
It seems that the algebra $\fa$ is the first Lie algebra having simple weight modules
with infinite dimensional weight spaces for which classification of simple weight modules is completed.
As an application we classify all simple weight $\mathcal{S}$-modules that have a simple $\fa$-submodule.
This provides many new examples of simple weight $\S$-modules (for other examples, see \cite{D,LMZ}).

The paper is organized as follows. We start with some preliminaries in Section~\ref{s2}. In Section~\ref{s3},
by embedding our algebras into the first Weyl algebra we obtain a very good presentation for the centralizer
$U_0$ of the Cartan subalgebra in the universal enveloping algebra $U(\fa)$. This is the key observation needed
to classify all simple $U_0$-modules. As a consequence, we obtain our main result, Theorem~\ref{thmmain},
which provides a classification of all simple weight modules over $\fa$. It turns out that there are four
classes of such simple modules, the first three classes consist of certain weight modules with
finite dimensional (in fact one-dimensional) weight spaces, while the last class consists of modules
for which all non-zero weight spaces are infinite-dimensional. As a bonus, in Section~\ref{s4} we
classify all simple weight $\S$-modules that have a simple $\fa$-submodule, see Theorem~\ref{thmtwo}.

\section{Preliminaries}\label{s2}

In this paper, we denote by $\Z$, $\N$, $\Z_+$ and $\C$ the sets of integers, positive integers, nonnegative
integers and complex numbers, respectively. All vector spaces and Lie algebras are over $\C$. For a Lie algebra
$\mathfrak{g}$ we denote by $U(\mathfrak{g})$ its universal enveloping algebra. We write $\otimes$ for
$\otimes_{\mathbb{C}}$.

The {\em extended Schr\"odinger algebra} $\mathcal{S}$ in $(1+1)$-dimensional space-time
is a complex Lie algebra with  a basis $\{ f,q,h,e,p,z\}$ and the Lie bracket given as follows:
\begin{equation}
\label{commrelations}
\begin{array}{lll}
\left[h,e\right]=2e, & \left[e,f\right]=h,  & \left[h,f\right]=-2f,\\
\left[e,q\right]=p, &  \left[e,p\right]=0, & \left[h,p\right]=p,\\
\left[f, p\right]=q, &\left[f,q\right]=0, &\left[h,q\right]=-q,\\
\left[p,q\right]=z. & [z,\mathcal{S}]=0.
\end{array}
\end{equation}
It is easy to see that the subspace $\fa$ of $\mathcal{S}$ spanned by $\{e,h,p,q,z\}$ is a Lie subalgebra
of $\mathcal{S}$. This subalgebra is called  the {\em ageing algebra}.
The elements $h$ and $z$ span a Cartan subalgebra $\mathfrak{h}$ while the elements $p$, $q$ and $z$ span a
copy of the Heisenberg algebra $\H$. We also denote by $\mathfrak{n}$ the subalgebra spanned by
$h$, $z$, $e$ and $p$. The subalgebra of $\mathcal{S}$ spanned by $e$, $f$ and $h$ is isomorphic to the classical
algebra $\sl_2$ and will be identified with the latter.

By Schur's lemma, the element $z$ acts as a scalar on any simple module over $\fa$ or $\mathcal{S}$. Since we are concerned with simple modules only, we will assume that  $z$ acts as a scalar $\dot{z}$ on any module in this paper. For a module
$V$   we denote by $\supp(V)$ the set of $h$-eigenvalues on $V$ and will
call these eigenvalues \emph{weights}. All $h$-eigenvectors will be called \emph{weight vectors}. For a weight
$\dot{h}$ we denote by $V_{\dot{h}}$ the corresponding {\em weight space}, that is the space of all $h$-eigenvectors
with eigenvalue $\dot{h}$. If $V$ is a simple weight $\fa$- or $\mathcal{S}$-module, then, as usual, we have
$\supp(V)\subset \dot{h}+\mathbb{Z}$ for any $\dot{h}\in \supp(V)$.

Let $U_0=\{x\in U(\fa) | [h,x]=0\}$ be the centralizer of $\mathfrak{h}$ in $U(\fa)$. Then, as usual, for every
simple weight $\fa$-module $V$ and any $\lambda\in\supp(V)$ the $U_0$-module $V_{\lambda}$ is simple.

Each $(\dot{z}, \dot{h})\in \C^2$ defines the one-dimensional $\mathfrak{n}$-module $\C w$ with the action
$h w=\dot{h} w$, $z w=\dot{z} w$ and $pw=ew=0$. Using it we define, as usual, the {\em Verma} $\fa$-module
$$M_{\fa}(\dot{z},\dot{h})=\Ind_{\mathfrak{n}}^{\fa} \C w.$$
Denote by $\bar{M}_{\fa}(\dot{z},\dot{h})$ the unique simple quotient of $M_{\fa}(\dot{z},\dot{h})$
(which exists by standard arguments, see e.g. \cite[Chapter~7]{Di}). Similarly we may
define the Verma  modules $M_{\mathcal{S}}(\dot{z}, \dot{h})$, $M_{\H}(\dot{z})$, $M_{\sl_2}(\dot{h})$
and the corresponding simple quotients $\bar{M}_{\mathcal{S}}(\dot{z}, \dot{h})$, $\bar{M}_{\H}(\dot{z})$, $\bar{M}_{\sl_2}(\dot{h})$ over $\mathcal{S}$, $\H$, and $\sl_2$, respectively. Analogously one defines the
{\em lowest weight Verma modules} $M^-_{\mathcal{S}}(\dot{z}, \dot{h})$, $M^-_{\fa}(\dot{z},\dot{h})$, $M^-_{\H}(\dot{z})$
and $M^-_{\sl_2}(\dot{h})$ and their corresponding simple quotients $\bar{M}^-_{\mathcal{S}}(\dot{z}, \dot{h})$, $\bar{M}^-_{\fa}(\dot{z},\dot{h})$, $\bar{M}^-_{\H}(\dot{z})$ and $\bar{M}^-_{\sl_2}(\dot{h})$.

\section{Simple weight modules over $\fa$}\label{s3}

We start by recalling some results from \cite{B} adjusted to our setup. Let $\mathcal{K}$ be the associative algebra
$\C(t)[s]$ where $st-ts=1$. The algebra $\mathcal{K}$ is a non-commutative  principal left and right
ideal domain, in fact, it is an  Euclidean domain (both left and right). The space $\C(t)$ becomes a
faithful $\C(t)[s]$-module by defining the action of $\C(t)$ via multiplication and the action of
$s$ via $\frac{d}{d t}$. Consider the subalgebra $R_0$
of $\mathcal{K}$ generated by $t$ and $t^2s$, and  the
subalgebra $R_{1}$ of $\mathcal{K}$ generated by $t$ and $ts$.
For convenience of later use we set $R_{\dot z}:=R_1$ for  $\dot{z}\in \C^*$.

\begin{lem}\label{lem1} Let $\dot z\in\C$.
{\hspace{2mm}}

\begin{enumerate}[$($a$)$]
\item\label{lem1.1} Let $\a=\sum_{j=0}^n \a_j s^j$, where $ \a_j\in \C(t)$ with $\a_0=1$,
be an irreducible element in $\mathcal{K}$. Then we have the following:
\begin{enumerate}[$($i$)$]
\item\label{lem1.1.1}
If the rational function $\a_j$ has a zero at $0$ of order at least $j+1$, then
$R_1/(R_1\cap \mathcal{K}\a)$ is a simple $R_1$-module. Up to isomorphism, every
$\C[t]$-torsion-free simple $R_1$-module arises in this way.
\item\label{lem1.1.2}
If the rational function $\a_j$ has a zero at $0$ of order at least $2j+1$, then
$R_0/(R_0\cap \mathcal{K}\a)$ is a simple $R_0$-module. Up to isomorphism, every
$\C[t]$-torsion-free simple $R_0$-module arises in this way.
\end{enumerate}
\item\label{lem1.2}
Any simple quotient of the $R_{\dot z}$-module $R_{\dot z}/R_{\dot z} t$ is 1-dimensional.
\item\label{lem1.3}
For any $\lambda\in \C^*$, the $R_{dot z}$-module $R_{\dot z}/R_{\dot z} (t-\lambda)$ is simple.
\item\label{lem1.4}
Let $V$ be a simple $\C[t]$-torsion-free $R_{\dot z}$-module. Then $t$ acts bijectively on $V$.
\end{enumerate}
\end{lem}

\begin{proof}
Claim~\eqref{lem1.1.1} is \cite[Theorem~2]{B}. Claim~\eqref{lem1.1.2} follows from \cite[Theorem~4.3]{B}.

To prove claim~\eqref{lem1.2}, note that  $R_{\dot z} t$ is a two-sided ideal of $R_{\dot z}$ for any value
of $\dot z$.  The algebra $R_{\dot z}/R_{\dot z} t$ is easily checked to be the commutative algebra
$\C[ts]$ if $\dot z\ne0$. Similarly, if $\dot z=0$, we have that $R_{\dot z}/R_{\dot z} t\cong \C[t^2s]$.
All simple modules  over the latter commutative algebras  have dimension one.

To prove claim~\eqref{lem1.3}, observe that $\{(ts)^i|i\in\Z_+\}$ is a basis in
$R_{\dot z}/R_{\dot z} (t-\lambda)$ if $\dot z\ne0$. Moreover, if $\dot z=0$, then
$\{(t^2s)^i|i\in\Z_+\}$ is a basis in the quotient $R_{\dot z}/R_{\dot z} (t-\lambda)$. Using this
it is easy to verify that the $R_{\dot z}$-module $R_{\dot z}/R_{\dot z} (t-\lambda)$ is simple.

Claim~\eqref{lem1.4} follows from the fact that $tV$ is an $R_{\dot z}$-submodule of $V$ and hence is equal to $V$.
\end{proof}

Next we characterize the associative algebra $U_0$. Denote by $\mathbf{A}$ the first Weyl algebra, which we
realize as the unital subalgebra of the algebra of all linear operators on $\mathbb{C}[t]$ generated by the
linear operator $\frac{d} {d t}$ and the linear operator of multiplication by $t$ (which we will denote simply by $t$).
Alternatively, one can also think
of the algebra $\mathbf{A}$ as the unital subalgebra of the algebra $\mathcal{K}$ generated by $t$ and $s$.
In particular, we can view $R_0$ as a subalgebra of $\mathbf{A}$ generated by $t$ and $t^2\frac{d} {d t}$.
Similarly, we can view $R_{1}$ as a subalgebra of $\mathbf{A}$ generated by $t$ and $t\frac{d} {d t}$.

\begin{lem}\label{lem2}
{\hspace{2mm}}

\begin{enumerate}[$($a$)$]
\item\label{lem2.1} The associative algebra  $U_0$ is generated by $eq^2$, $pq$, $h$, and $z$ and
$$\{(eq^2)^{i_1}(pq)^{i_2}h^{i_3}z^{i_4}|i_1,i_2,i_3,i_4\in \Z_+\}$$
is a basis of $U_0$. In particular, for any $\dot{z},\dot{h}\in \C$, the image of
$\{(eq^2)^{i_1}(pq)^{i_2}|i_1,i_2\in \Z_+\}$ in $U_0/\langle h-\dot{h}, z-\dot{z}\rangle$ forms a basis there.
\item\label{lem2.2}  For any $\dot{z},\dot{h}\in \C$, there is a unique homomorphism of associative algebras
$\phi_{\dot{z},\dot{h}}: U_0\rightarrow \mathbf{A}$ such that $\phi_{\dot{z},
\dot{h}}(z)=\dot{z}$, $\phi_{\dot{z},\dot{h}}(h)=\dot{h}$ and
\begin{enumerate}[$($i$)$]
\item\label{lem2.2.1}
 $\phi_{0, \dot{h}}(pq)=t,\,\, \phi_{0, \dot{h}}(eq^2)=t^2\frac{d}{d
t}$ if $\dot{z}=0$;
\item\label{lem2.2.2} $\phi_{\dot{z}, \dot{h}}(pq)=2\dot{z} t\frac{d}
{d t},\,\, \phi_{\dot{z}, \dot{h}}(eq^2)=t+\dot{z}t\frac{d}{d
t}+2\dot{z}(t\frac{d} {d t})^2$  if $\dot{z}\neq 0$.
\end{enumerate}
\item\label{lem2.3}
We have $\phi_{\dot{z},\dot{h}}(U_0)=R_{\dot z}$ and
$U_0/\langle h-\dot{h}, z-\dot{z} \rangle\cong R_{\dot z}$.
\item\label{lem2.4}
The isomorphism in \eqref{lem2.3} induces a natural bijection between isomorphism classes of simple
$U_0 $-modules on which $h$ acts as $\dot{h}$ and $z$ acts as $\dot{z}$ and isomorphism classes of
simple $R_{\dot z}$-modules.
\end{enumerate}
\end{lem}

\begin{proof}
To prove claim~\eqref{lem2.1} we argue that, by the PBW Theorem, the algebra $U_0$ has a basis
\begin{gather*}
\{e^{i_1}q^{2i_1+i_2}p^{i_2}z^{i_3}h^{i_4}|i_1,i_2,i_3,i_4\in \Z_+\}=\\
\{(e^{i_1}q^{2i_1})(q^{i_2}p^{i_2})z^{i_3}h^{i_4}|i_1,i_2,i_3,i_4\in \Z_+\}.
\end{gather*}

{\bf Step 1}: The element $q^{i}p^{i}$ can be written as a linear
combination of elements of the form $(pq)^jz^{i-j}$ for $j\in\Z_+$.

This follows by induction on $i$ from the following computation:
$$q^{i+1}p^{i+1}=q^{i+1}pp^{i}=pqq^{i}p^{i}-(i+1)q^{i}p^{i}z.$$

{\bf Step 2}: The element $e^iq^{2i}$ can be written as a linear
combination of elements of the form  $(eq^2)^k(pq)^jz^{i-k-j}$ for $k,j\in\Z_+$ with $i\ge k+j$.

This follows by induction on $i$ from the following computation:
\begin{multline*}
e^{i+1}q^{2(i+1)}=ee^{i}qq^{2i+1}=eqe^{i}q^{2i+1}+ie^{i}pq^{2i+1}=\\
eq^2e^{i}q^{2i}+2ie^{i}pq^{2i+1}=(eq^2)(e^{i}q^{2i})+2i((e^{i}q^{2i})(pq)+2ie^{i}q^{2i}z).
\end{multline*}

Claim~\eqref{lem2.1} follows easily from Steps~1 and 2.

To prove claim~\eqref{lem2.2}, we only need to check
$$\phi_{\dot{z},\dot{h}}([eq^2,pq])=[\phi_{\dot{z},\dot{h}}(eq^2),
\phi_{\dot{z},\dot{h}}(pq)]$$ for any $\dot{z}$ and $\dot{h}$. Note that
\begin{displaymath}
\begin{array}{rcl}
[eq^2,pq]&=&[eq^2,p]q+p[eq^2,q]\\&=&-2zeq^2+p(pq)q\\ &=&-2zeq^2+zpq +(pq)^2.
\end{array}
\end{displaymath}
Therefore we have $[\phi_{0,\dot{h}}(eq^2),\phi_{0,\dot{h}}
(pq)]=[t^2\frac{d}{dt},t]=t^2$ while
$$\phi_{0,\dot{h}}([eq^2,pq])=\phi_{0,\dot{h}}(-2zeq^2+zpq+(pq)^2)=t^2,$$
which implies \eqref{lem2.2.1}.

Similarly, for $\dot{z}\ne 0$ we have
\begin{displaymath}
\begin{array}{rcl}
\phi_{\dot{z},\dot{h}}([eq^2,pq])&=&\phi_{\dot{z},\dot{h}}
((-2zeq^2+zpq+(pq)^2))\\&=&-2\dot{z}(t+\dot{z}t\frac{d}{d t}+2\dot{z}(t\frac{d} {d t})^2)+\dot{z}2\dot{z}
 t\frac{d} {d t}+(2\dot{z} t\frac{d} {d t})^2\\&=&-2\dot{z}t,
\end{array}
\end{displaymath}
while
\begin{displaymath}
[\phi_{\dot{z},\dot{h}}(eq^2),\phi_{\dot{z},\dot{h}}(pq)]=
[t+\dot{z}t\frac{d}{d t}+2\dot{z}(t\frac{d} {d t})^2,2\dot{z} t\frac{d}{d t}]=-2\dot{z}t,
\end{displaymath}
which implies \eqref{lem2.2.2}.

Claim~\eqref{lem2.3} follows from claims~\eqref{lem2.1} and \eqref{lem2.2}.
Claim~\eqref{lem2.4} follows from claim~\eqref{lem2.3}.
\end{proof}

Now we address the structure of Verma modules over $\fa$.

\begin{lem}\label{lem3}
{\hspace{2mm}}

\begin{enumerate}[$($a$)$]
\item\label{lem3.1} Let $\dot{z},\dot{h}\in\C$. Then the $\fa$-module
$M_{\fa}(\dot{z},\dot{h})$ is simple if and only if $\dot{z}\ne 0$.
\item\label{lem3.2} Let $\dot{h}\in\C$.  Then we have $\dim \bar{M}_{\fa}(0,\dot{h})=1$.
\end{enumerate}
\end{lem}

\begin{proof}
If $\dot{z}\ne 0$, then the module $M_{\fa}(\dot{z},\dot{h})$ is obviously simple already as an
$\mathcal{H}$-module, which implies \eqref{lem3.1}. If $\dot{z}= 0$, we have a simple $\fa$-module
$\C v$ with action $ev=pv=qv=zv=0$ and $h v=\dot{h} v$. By the universal property of Verma modules,
${M}_{\fa}(0,\dot{h})$ surjects onto $\C v$ and hence  $\bar{M}_{\fa}(0,\dot{h})=\C v$,
which implies \eqref{lem3.2}.
\end{proof}

\begin{rem}\label{remn2}
{\rm
Lemma~\ref{lem3} describes all simple highest weight $\fa$-mo\-du\-les.  Let $V$ be a simple highest weight $\fa$-module.
Consider the decomposition $V=\oplus_{\dot{h}\in\mathbb{C}}V_{\dot{h}}$ and note that all $V_{\dot{h}}$
are finite dimensional. As usual, the space $V^{\star}=\oplus_{\dot{h}\in\mathbb{C}}\mathrm{Hom}_{\C}(V_{\dot{h}},\C)$
has the natural structure of an $\fa$-module defied using the canonical involution $x\mapsto -x$, $x\in \fa$.
The module $V^{\star}$ is  a simple lowest weight $\fa$-module and the correspondence
$V\mapsto V^{\star}$ is a bijection between the sets of isomorphism classes of simple highest weight and
simple lowest weight modules.
}
\end{rem}

Before constructing some other weight $\fa$-modules, we have to define some automorphisms of the associative
algebras $U_0$ and $R_{\dot z}$.

\begin{lem}\label{lem4}
For any $i\in \Z$ there is a unique $\tau_{i}\in \Aut(U_0)$ such that
\begin{gather*}
\tau_i(pq)=pq+iz, \,\,\,\tau_i(h)=h-i, \,\,\,\tau_i(z)=z,\\
\tau_i(eq^2)=eq^2+ipq+\frac{i(i+1)}{2}z.
\end{gather*}
\end{lem}

\begin{proof}  From Lemma 2.1(a), we know that  $U_0$ is a PBW-algebra on the generating set $\{eq^2,pq, h, z\}$.  One can find the concept of  PBW-algebras in  \cite{BTLC}. Then  we need only to verify the  relations on the generating elements $eq^2$, $pq$, $h$, and $z$.
We have
\begin{displaymath}
\begin{array}{rcl}
[\tau_i(eq^2),\tau_i(pq)]&=&[eq^2+ipq+\frac{i(i+1)}{2}z,pq+iz]\\&=&[eq^2,pq]\\
&=&-2zeq^2+zpq+(pq)^2
\end{array}
\end{displaymath}
while
\begin{displaymath}
\begin{array}{rcl}
\tau_i([eq^2,pq])&=&\tau_i(-2zeq^2+zpq+(pq)^2)\\&
=&-2z(eq^2+ipq+\frac{i(i+1)}{2}z)+z(pq+iz)+(pq+iz)^2\\&=&-2zeq^2+zpq+(pq)^2.
\end{array}
\end{displaymath}
Thus $[\tau_i(eq^2),\tau_i(pq)]=\tau_i([eq^2,pq])$. All other relations are checked similarly.
\end{proof}

The proof of the next lemma is a straightforward computation which is left to the reader.
Here we use the embedding of  $R_{\dot z}$ in the first Weyl algebra $\mathbf{A}$ as in the paragraph before Lemma 2.

\begin{lem}\label{lem5}
For any $i\in \Z$ and $\dot{z}\in \C$ there is a unique $\tau_{i,\dot{z}}\in \Aut(R_{\dot z})$ such that
\begin{gather*}
\tau_{i, 0}(t)=t,\,\,\,\,\,\, \tau_{i, 0}(t^2\frac{d}{d t})=t^2\frac{d}{d t}+it \,\,\,\,\text{if} \,\,\,\,\dot{z}= 0;\\
\tau_{i,\dot{z}}(t)=t,\,\,\,\,\,\, \tau_{i,\dot{z}}( t\frac{d} {d t})=t\frac{d} {d t}+\frac{i}{2}
\,\,\,\,\,\,\text{if}\,\, \dot{z}\ne 0.
\end{gather*}
\end{lem}

The next step is to construct some weight $\fa$-modules using $R_{\dot z}$-modules.  From now on in this section
$N$ denotes an $R_{\dot z}$-module. We define the $\fa$-module $N_{\dot{z},\dot{h}}$ as follows: as a vector space
we set $N_{\dot{z},\dot{h}}=N\otimes \mathbb{C}[x,x^{-1}]$ and then for $v\in N$ define
\begin{equation}\label{A-1}q(v\otimes x^i)=v\otimes x^{i+1},\,\,\,\,\,\, p(v\otimes x^i)
=(\phi_{\dot{z},\dot{h}}(pq+(i-1)z) v)\otimes x^{i-1},\end{equation}
\begin{equation}\label{A-2}e (v\otimes x^{i})=(\phi_{\dot{z},
\dot{h}}(eq^2+(i-2)pq+\frac{(i-1)(i-2)}{2}z)v)\otimes x^{i-2},\end{equation}
\begin{equation}\label{A-3}h (v\otimes x^i)=((\dot{h}-i) v)\otimes x^i,\,\,\,\,\,\,
z(v\otimes x^i)=(\dot{z} v)\otimes x^i. \end{equation}
By a straightforward computation we get the following:

\begin{lem}\label{lem6}
{\hspace{2mm}}

\begin{enumerate}[$($a$)$]
\item\label{lem6.1}
Formulae \eqref{A-1}, \eqref{A-2}  and \eqref{A-3} define on $N_{\dot{z},\dot{h}}$  the structure of an $\fa$-module.
\item\label{lem6.2}
For $\dot{z}=0$ the action in \eqref{A-1}, \eqref{A-2}  and \eqref{A-3} reads as follows
{\small
\begin{equation}\label{A-01}
\left\{\aligned&h (v\otimes x^i)=(\dot{h}-i) v\otimes x^i,\,\,\,\,z(v\otimes x^i)=0, \\
 & q (v\otimes x^i)=v\otimes x^{i+1},\,\,\,\, p(v\otimes x^i)=(t v)\otimes x^{i-1},\\
& e (v\otimes x^{i})=((t^2\frac{d}{d t}+(i-2)t)v)\otimes x^{i-2}.
\endaligned\right.
\end{equation}
}
\item\label{lem6.3}
For $\dot{z}\neq 0$ the action in \eqref{A-1}, \eqref{A-2}  and \eqref{A-3} reads as follows
{\footnotesize
\begin{equation}\label{A-11}
\left\{\aligned&h (v\otimes x^i)=(\dot{h}-i) v\otimes x^i,\,\,\,\,z(v\otimes x^i)=\dot{z} v\otimes x^i, \\
 & q (v\otimes x^i)=v\otimes x^{i+1},\,\,\,\,
 p(v\otimes x^i)=((2\dot{z} t\frac{d}{d t}+(i-1)\dot{z}) v)\otimes x^{i-1},\\
&e (v\otimes x^{i})=((t+(2i-3)\dot{z}t\frac{d}{d t}+2\dot{z}(t\frac{d}{d t})^2+\frac{(i-1)(i-2)}{2}\dot{z})v)
\otimes x^{i-2}.\\
\endaligned
\right.
\end{equation}
}
\end{enumerate}
\end{lem}

For an associative algebra $A$, an $A$-module $V$ and $\sigma\in \Aut(A)$, we denote by $V^{\sigma}$
the module obtained from $V$ via twisting the action of $A$ by $\sigma$, that is $a\cdot v=\sigma(a)v$
for $a\in A$ and $v\in V$.

For each $i\in \Z$ the space $N\otimes x^i$ is naturally a $U_0$-module. We may even make
$N\otimes x^i$ into an $R_{\dot z}$-module as follows:

If $\dot{z}=0$, then for any $ v\in N$, we set
\begin{equation} \label{z0} \aligned
&t(v\otimes x^i)=(pq)(v\otimes x^i)=(tv)\otimes x^i,\\
&(t^2\frac d{dt})(v\otimes x^i)=(eq^2)(v\otimes x^i)=((t^2\frac d{dt}+it)v)\otimes x^i.
\endaligned
\end{equation}

If $\dot{z}\ne0$, then for any $ v\in N$, we set
\begin{equation} \label{z1} \aligned
&t(v\otimes x^i)=(eq^2-\frac{pq}{2}-2\dot{z}(\frac{pq}{2\dot{z}})^2)(v\otimes x^i)=(tv)\otimes x^i,\\
&(t\frac d{dt})(v\otimes x^i)=(\frac{pq}{2\dot{z}})(v\otimes x^i)=((t\frac
d{dt}+\frac i2)v)\otimes x^i.
\endaligned
\end{equation}

The following standard result asserts that any simple weight module is
uniquely determined by any of its nonzero weight spaces.

\begin{lem} \label{lemma-2.4}
Let $\mathfrak{g}$ be $\fa$ or $\mathcal{S}$. Let $V$ and $W$ be simple weight $\mathfrak{g}$-modules such that
for some $\dot{h}\in \supp(V)$ the $U(\mathfrak{g})_0$ modules $V_{\dot{h}}$ and $W_{\dot{h}}$ are isomorphic.
Then $V\cong W$.
\end{lem}

\begin{proof}
Set $N=V_{\dot{h}}=W_{\dot{h}}$ and write
$$U(\mathfrak{g})=\bigoplus_{i\in\Z}U_i\quad\text{where}\quad
U_i=\{x\in U(\mathfrak{g}) |[h,y]=iy\}.$$
Every $U_i$ is a $U(\mathfrak{g})_0\text{-}U(\mathfrak{g})_0$-bimodule.
Consider the induced module
\begin{displaymath}
M:=\mathrm{Ind}_{U(\mathfrak{g})_0}^{U(\mathfrak{g})}N=
U(\mathfrak{g})\bigotimes_{U(\mathfrak{g})_0} N=
\bigoplus_{i\in\Z}({U_i}\otimes_{U(\mathfrak{g})_0} N).
\end{displaymath}
Then, by adjunction, both $V$ and $W$ are simple quotients of  $M$. The module $M$ is, clearly,  a weight
$\mathfrak{g}$-module with $M_{\dot{h}}\cong N$, the latter being a simple $U(\mathfrak{g})_0$-module,
which implies, using standard arguments, that  $M$ has a  unique maximal submodule $K$ satisfying $K_{\dot{h}}=0$
and thus the unique corresponding simple quotient $M/K$. As both $V$ and $W$ have to be quotients of $M/K$
by adjunction, we get $V\cong W$.
\end{proof}

Now we list some properties of the $\fa$-module $N_{\dot{z},\dot{h}}$.

\begin{lem}\label{lemma-2.5}
{\hspace{2mm}}

\begin{enumerate}[$($a$)$]
\item\label{lemma-2.5.1}
For every $i\in\Z$ the $U_0$-modules $N\otimes x^i$ and $N^{\tau_i}$ are isomorphic.
\item\label{lemma-2.5.2}
The $\fa$-module  $N_{\dot{z},\dot{h}}$ is simple if and only if $N$ is a simple $R_{\dot{z}}$-mo\-dule and
one of the following conditions is satisfied:
\begin{enumerate}[$($i$)$]
\item\label{lemma-2.5.2.1}
$\dot{z}=0$ and $(\C t+\C t^2\frac{d}{d t})N\ne 0$;
\item\label{lemma-2.5.2.2}
$\dot{z}\ne 0$ and $\tau_{i,\dot{z}}(\C t+\C t\frac{d}{d t})N\ne0$ for all $i\in \Z$.
\end{enumerate}
\item\label{lemma-2.5.3}
Let $\dot{z},\dot{h}, \dot{z}',\dot{h}'\in\C$. Let further $N$ be an $R_{\dot{z}}$-module
and $N'$ be an $R_{\dot{z}'}$-module. Assume that $N_{\dot{z},\dot{h}}$ and $N'_{\dot{z}',\dot{h}'}$ are simple.
Then we have $N_{\dot{z},\dot{h}}\cong
N'_{\dot{z}',\dot{h}'}$  if and only if $\dot{z}=\dot{z}'$,
$i:=\dot{h}-\dot{h}'\in \Z$ and $N'\cong N^{\tau_{i,\dot{z}}}$ as
$R_{\dot{z}}$-modules.
\end{enumerate}
\end{lem}

\begin{proof}
Let $\psi:N\otimes x^i\rightarrow N^{\tau_i}$ be the linear map defined as
$\psi(v\otimes x^i)=v$ for all $v\in A$. It is straightforward to verify that
$$\aligned &\psi(pq (v\otimes x^i))=\psi(((pq+i\dot{z})v)\otimes x^i)=(pq+iz)v=\tau_i(pq) v,\\
&\psi((eq^2) (v\otimes x^i))=\psi((eq^2+ipq+\frac{i(i+1)}{2}\dot{z})v\otimes x^i)\\&=(eq^2+ipq+\frac{i(i+1)}
{2}\dot{z})v=\tau_i(eq^2)\psi(v\otimes x^i),\\
&\psi(h (v\otimes x^i))=\psi((\dot{h}-i)v\otimes x^i)=(\dot{h}-i)v=\tau_i(h) \psi(v\otimes x^i),\\
&\psi(z (v\otimes x^i))=\psi(\dot{z}v\otimes x^i)=\tau_i(z) \psi(v\otimes x^i).\endaligned$$
This implies claim~\eqref{lemma-2.5.1}.

To prove claim~\eqref{lemma-2.5.2}, we observe that simplicity of $N_{\dot{z},\dot{h}}$ clearly requires simplicity of
$N$. Now suppose that $N$ is a simple $R_{\dot{z}}$-module. We study simplicity of $N_{\dot{z},\dot{h}}$ using
a case-by-case analysis.

{\bf Case~1.} Assume $\dot{z}=0$ and $(\C t+\C t^2\frac{d}{d t})N= 0$.

In this case, from \eqref{A-01} we get $e N_{0,\dot{h}}=pN_{0,\dot{h}}=0$. Hence each
nonzero weight element of $N_{0,\dot{h}}$ generates a proper highest weight submodule.

{\bf Case~2.} Assume $\dot{z}=0$ and $(\C t+\C t^2\frac{d}{d t})N\ne 0$.

Since $N$ is a simple $R_{\dot{z}}$-module, we have $(\C t+\C t^2\frac{d}{d t})v\ne 0$ for any nonzero $v\in N$.
From \eqref{A-01} it follows that for each $i\in \Z$ we either have $p(v\otimes x^i)\ne 0$ or $e(v\otimes x^{j})\ne 0$.
As the action of $q$ on $N_{0,\dot{h}}$ is injective, it follows that any nonzero submodule $V$ of $N_{0,\dot{h}}$
has support $\dot{h}+\Z$.  Now from claim~\eqref{lemma-2.5.1} we get $V=N_{0,\dot{h}}$, that is
any nonzero submodule $V$ coincides with $N_{0,\dot{h}}$ and thus $N_{0,\dot{h}}$ is simple.

{\bf Case~3.} Assume $\dot{z}\ne 0$ and $\tau_{i,\dot{z}}(\C t+\C t\frac{d}{d t})N=0$ for some $i\in \Z$.

In this case, we have $N=\C v$ with $t v=0, (t\frac{d} {dt})v=-\frac{i}{2}v$. From \eqref{A-11}, we have
$p(v\otimes x^{i+1})=e(v\otimes x^{i+1})=0$ and hence $v\otimes x^{i+1}$ generates a proper highest weight submodule
of $N_{\dot{z},\dot{h}}$. Therefore $N_{\dot{z},\dot{h}}$ is not simple.

{\bf Case~4.} Assume $\dot{z}\ne 0$ and $\tau_{i,\dot{z}}(\C t+\C t\frac{d}{d t})N\ne0$ for all $i\in \Z$.

Since $N$ is a simple $R_{\dot{z}}$-module, have $\tau_{i,\dot{z}}(\C t+\C t\frac{d}{d t})v\ne 0$ for any nonzero
$v\in N$. From \eqref{A-11} it follows that for each $i\in \Z$ we have either $p (vx^{i+1})\ne 0$ or
$e(vx^{i+2})\ne 0$ for some $v\in N$. Similarly to Case~2 we deduce that $N_{\dot{z},\dot{h}}$ is simple.
This completes the proof of claim~\eqref{lemma-2.5.2}.

To prove claim~\eqref{lemma-2.5.3}, first assume $N_{\dot{z},\dot{h}}\cong N'_{\dot{z}',\dot{h}'}$. Then we
have $i=h-h'\in \Z$, $\dot{z}=\dot{z}'$ and  $N'\cong N\otimes x^i$ as $U_0$-modules. From \eqref{z0} and
\eqref{z1} it follows that $A'\cong N^{\tau_{i,\dot{z}}}$ as $R_{\dot{z}}$-modules.

Now suppose that $\dot{z}=\dot{z}'$, $i=\dot{h}-\dot{h}'\in \Z$ and $N'\cong N^{\tau_{i,\dot{z}}}$ as
$R_{\dot{z}}$-modules. From \eqref{z0} and \eqref{z1} it follows that $N'\cong N\otimes x^i$ as $U_0$-modules.
From Lemma \ref{lemma-2.4} we thus get $N_{\dot{z},\dot{h}}\cong N'_{\dot{z}',\dot{h}'}$. This completes the proof.
\end{proof}

\begin{rem}\label{remn1}
{\rm
Let $N$ be a simple $R_{\dot{z}}$-module. If we have $\dot{z}=0$ and $(\C t+\C t^2\frac{d}{d t})N=0$,
then $\dim N=1$. If $\dot{z}\ne 0$ and $\tau_{i,\dot{z}}(\C t+\C t\frac{d}{d t})N=0$ for some $i\in \Z$, then $\dim N=1$.
Thus, if $N$ is an infinite dimensional simple $R_{\dot{z}}$-module, then $N_{\dot{z},\dot{h}}$ is a
simple weight $\fa$-module.
}
\end{rem}

Now we are ready to prove our main result on classification of all
simple weight $\fa$-modules.

\begin{thm}\label{thmmain}
Each simple weight $\fa$-module is isomorphic to one of the following simple modules:
\begin{enumerate}[$($i$)$]
\item\label{thmmain.1}
A simple highest or lowest weight module.
\item\label{thmmain.2}
The module $Q(\lambda,\dot{h})=\C[x,x^{-1}]$, where $\dot{h}\in \C$ and $\lambda\in \C^*$, with the action given by:
$$z x^i=p x^i=0,\,\,\,\, qx^i=x^{i+1},\,\,\,\,ex^i=\lambda x^{i-2},\,\,\,\,h x^i=(\dot{h}-i)x^i.$$
\item\label{thmmain.3}
The module  $Q'(\dot{z},\dot{h},\lambda)=\C[x,x^{-1}]$, where $\dot{h}\in \C,\dot{z}\in \C^*$ and $\lambda \in
\C\backslash \Z$, with the action given by:
$$\aligned &q x^i=x^{i+1},\,\,\,\, p x^i=\dot{z}(\lambda+i)x^{i-1},\,\,\,\,z x^i=
\dot{z} x^i,\\& h x^i=(\dot{h}-i)x^i,\,\,\,\, e
x^i=\frac{\dot{z}}{2}(\lambda+i)(\lambda+i-1)x^{i-2}.\endaligned$$
\item\label{thmmain.4}
The module $N_{\dot{z},\dot{h}}$, where $N$ is an infinite dimensional simple $R_{\dot{z}}$-module.
\end{enumerate}
\end{thm}

\begin{proof}
It is straightforward to verify that $Q(\lambda,\dot{h})$ in \eqref{thmmain.2} and $Q'(\dot{z},\dot{h},\lambda)$
in \eqref{thmmain.3} are simple $\fa$-modules. From Remark~\ref{remn1} we know that
the module $N_{\dot{z},\dot{h}}$ in \eqref{thmmain.4} is a simple $\fa$-module.

Let $V$ be any simple weight $\fa$-module. Assume that $\dot{h}\in \supp(V)$ and that $z$ acts on $V$ as the
scalar $\dot{z}$. Then $V_{\dot{h}}$ is both, a simple $U_0$-module and a simple $R_{\dot{z}}$-module.
By Lemma~\ref{lem1} we have that $V_{\dot{h}}$ either is  a one-dimensional module with $tV_{\dot{h}}=0$ or
is an infinite dimensional module with $t$ acting bijectively on it. If $N=V_{\dot{h}}$ has
infinite dimension, then $V\cong N_{\dot{z},\dot{h}}$ by Lemma~\ref{lemma-2.4}. Therefore it remains to
consider the case when all nontrivial weight spaces of $V$ have dimension one.

If $V$ has a nonzero element annihilated by $q$, then $V$ is clearly a lowest weight module. Therefore
we may assume that $q$ acts injectively on $V$. Let $0\ne w\in V_{\dot{h}}$ and consider first
the case $\dot z=0$.

In this case, we have $tw=pqw=0$ and hence $pw=0$ since  the action of $q$ is injective. This implies that the
ideal $\C p+\C z$ of $\fa$ annihilates $V$. Let $\mathfrak{a}=\fa/(\C p+\C z)$. Then $V$ is a simple
$\mathfrak{a}$-module. Denote by $\bar{x}$ the image of $x\in \fa$ in $\mathfrak{a}$. Note that
$\bar{e}\bar{q}^2$ is a central element in $U(\mathfrak{a})$. By Schur's lemma, $\bar{e}\bar{q}^2$ acts
on $V$ as a scalar, say  $\lambda\in\mathbb{C}$. If $\lambda=0$, then $e (q^2w)=0$ and $V$ is a highest
weight module. If $\lambda\ne 0$, then it is easy to check that $V\cong Q(\lambda,\dot{h})$.

Finally, consider the case  $\dot z\ne 0$.  In this case,  $w$ generates an $\H$-submodule of $V$ with nonzero central charge
and one-dimensional weight spaces. Consequently, $V$ contains a simple $\C h+\H$ submodule $M$ with nonzero central
charge.  We make $M$ into an $\fa$-module by setting $ev=\frac{1}{2\dot{z}} p^2 v$ and denote the resulting module
by $M^{\fa}$.  Let $\C v$ be the one-dimensional trivial $\C h+\H$-module. Then
\cite[Lemma~8 and Theorem~7]{LZ1} yield that $V$ is isomorphic to a simple quotient of
$$\Ind_{\C  h+\H}^{\fa}(M\otimes \C v)\cong M^{\fa}\otimes \Ind_{\C
h+\H}^{\fa} \C v,$$ which is of the form $M^{\fa}\otimes X$, where $X$ is a simple quotient of
$\Ind_{\C  h+\H}^{\fa} \C v$.  Since $M^{\fa}\otimes X$ is a weight module, we get that $X$ is the
trivial module and hence $V\cong M^{\fa}$. This gives that
$V$ is either a highest weight module or is isomorphic to $Q'(\dot{z},\dot{h},\lambda)$ with
$\lambda\in\C\setminus\Z$. This completes the proof.
\end{proof}

The following result is a direct consequence of Theorem~\ref{thmmain}.

\begin{cor}\label{cor21}
Let $V$ be a simple weight $\fa$-module and $\dot{h}\in \supp(V)$. Then either
$\dim V_{\dot{h}+i}\le 1$ for all $i\in \Z$ or $\dim V_{\dot{h}+i}=\infty$ for
all $i\in \Z$.
\end{cor}

Next we present a nontrivial example of a simple weight $\fa$-module with infinite dimensional weight spaces.

\begin{exa}\label{ex22}
{\rm
Let $\a=1-t^3(t-1)\frac{d}{dt}\in \mathcal{K}$.
It is easy to see that $\alpha$ is irreducible in $\mathcal{K}$.
Then $N=R_0/(R_0\cap \mathcal{K}\a)$ is a simple $R_0$-module such that
$t$ acts bijectively on $N$. This implies that $N$ is also a simple
$R_1$-module and a simple module over the localization of $\mathbf{A}$ at $t$. In
fact, we have $N=(t-1)^{-1}\C[t,t^{-1}]$ with the following action:
$$\frac{d}{d t}\cdot g(t)=\frac{ d g(t)}{d t}+\frac{g(t)}{t^3(t-1)},\,\,\,\,
t \cdot g(t)=tg(t)\,\,\text{ for all } g(t)\in N. $$ Thus we have the simple
$\fa$-module $N_{\dot{z},\dot{h}}=(t-1)^{-1}\C[t,t^{-1},x,x^{-1}]$ with the action
\begin{gather*}
h:=\dot{h}-\partial_x,\,\,\,\,z:=0,\,\,\,\, q:=x,\,\,\,\, p:=t x^{-1},\\
e:=x^{-2}t\partial_t+\frac{1}{t(t-1)x^2}+tx^{-2}\partial_x-2tx^{-2}  .
\end{gather*}
if $\dot{z}=0$, and the action
\begin{gather*}
h :=\dot{h}-\partial_x,\,\,\,\,z:=\dot{z}, \,\,\,\,
  q:=x,\,\,\,\, p:=x^{-1}\dot{z}(2\partial_t+\frac{2}{t^2(t-1)}+\partial_x-1),\\
e:=\dot{z}x^{-2}\left(\frac{t}{\dot{z}}+(2\partial_x-3)(\partial_t+\frac{1}{t^2(t-1)})
+2\partial_t^2+\frac{4}{t^2(t-1)}\partial_t-\right.\\\left.
\hspace{4cm} \frac{2+2t+6t^2}{t^4(t-1)^2}
+\frac{(\partial_x-1)(\partial_x-2)}{2}\right),
\end{gather*}
if $\dot{z}\ne 0$ (here $\partial_x=x\frac{\partial}{\partial x}$ and $\partial_t=t\frac{\partial}{\partial t}$).
}
\end{exa}

\section{Simple weight $\mathcal{S}$-modules having a simple $\fa$-submodule}\label{s4}

In this section we classify all simple weight $\S$-modules that have a simple $\fa$-submodule.

Let $V$ be a simple weight $\fa$-module. We have the induced weight $\S$-module $H(V):=\Ind_{\fa}^{\S} V$ which
can be identified with $\C[f]\otimes V$ as a vector space. In this section we classify  all simple quotient
$\S$-modules of $H(V)$. From \cite[Corollary~8]{DLMZ}  we know that the center $Z(U(\S))$ of $U(\S)$ equals
$\C[z,c]$, where
\begin{equation}\label{eq55}
c:=(fp^2-eq^2-hpq)-\frac{1}{2}(h^2+h+4fe)z.
\end{equation}
Recall some simple $\S$-module constructed in \cite{LMZ} and \cite{D}. Let $V$ be any simple module over
$\H$ with nonzero central charge $\dot{z}$. The module $V$ becomes an $\S$-module by setting
\begin{equation}\label{eq57}
e v=\frac{1}{2\dot{z}} p^2 v,\,\,\,\,  f v=-\frac{1}{2 \dot{z}}q^2 v,\,\,\,\,  h v=(-\frac{pq}{\dot{z}}+\frac{1}{2})v,
\text{ where }  v\in V.
\end{equation}
The resulting module will be denoted by $V^{\S}$.

Any simple $\sl_2$-module $W$ becomes an $\S$-module by setting $\H W=0$. The resulting module will also
be denoted by $W^{\S}$.

Next let us define a class of simple weight $\S$-modules.  Let $N$ be an infinite dimensional simple
$R_{\dot{z}}$-module. Let $\mathbf{A}_{(t)}$ denote the localization of $\mathbf{A}$ at powers of $t$, which is actually the differential operator algebra $\C[t, t^{-1}, \frac d{dt}]$.
Then $N$ is also naturally an $\mathbf{A}_{(t)}$-module since $t$ acts bijectively on $N$, see
Lemma~\ref{lem1}\eqref{lem1.4}. Let
$\dot{c},\dot{z},\dot{h}\in \C$. Then we have the simple weight $\fa$-module
$N_{\dot{z},\dot{h}}$ given by Lemma~\ref{lem6}. We extend this $\fa$-module
$N_{ \dot{z},\dot{h}}$ to an $\S$-module $N_{\dot{c}, \dot{z},\dot{h}}$ as follows:
For $\dot{z}=0$ we set:
\begin{equation}\label{eq77}
\begin{array}{l}
h (v\otimes x^i)=(\dot{h}-i) v\otimes x^i,\,\,\,\, z(v\otimes x^i)=0, \\
q (v\otimes x^i)=v\otimes x^{i+1},\,\,\,\, p(v\otimes x^i)=(t v)x^{i-1},\\
e (v\otimes x^{i})=((t^2\frac{d}{d t}+(i-2)t)v)\otimes x^{i-2},\\
f(v\otimes x^i)=((\frac{d}{d t}+(\dot{h}-2 )t^{-1}+\dot{c}t^{-2})v)\otimes x^{i+2}.
\end{array}
\end{equation}
For $\dot{z}\neq 0$ we set:
\begin{equation}\label{eq75}
\begin{array}{l}
h (v\otimes x^i)=(\dot{h}-i) v\otimes x^i,\,\,\,\, z(v\otimes x^i)=\dot{z} v\otimes x^i,\\
q (v\otimes x^i)=vx^{i+1},\,\,\,\, p(v\otimes x^i)=((2\dot{z} t\frac{d}{d t}+(i-1)\dot{z}) v)\otimes x^{i-1},\\
e (v\otimes x^{i})=((t+(2i-3)\dot{z}t\frac{d}{d t}+
2\dot{z}(t\frac{d}{d t})^2+\frac{(i-1)(i-2)}{2}\dot{z})v)\otimes x^{i-2},\\
f(v\otimes x^i)=((-\frac{1}{2\dot{z}}-(\dot{h}-\frac{1}{2})\frac{d}{
dt}-t(\frac{d}{dt})^2-(\frac{(\dot{h}-1)(\dot{h}-2)}{4}+\frac{\dot{c}}{2\dot{z}})t^{-1})v)\otimes x^{i+2}.
\end{array}
\end{equation}

\begin{lem}\label{lem15}
Formulae~\eqref{eq77} and \eqref{eq75} indeed define the structure of an $\S$-module, moreover,
the element $c$ acts on $N_{\dot{c}, \dot{z},\dot{h}}$ as the scalar $\dot{c}$.
\end{lem}

\begin{proof}
If $\dot{z}=0$, then from \eqref{A-01} we know that $p=tx^{-1}$ which acts
bijectively on $N_{ \dot{z},\dot{h}}$. From \eqref{eq55} we obtain that
$f=(eq^2+hpq+\dot{c})p^{-2}$, that is $f(v\otimes x^i)=((\frac{d}{d
t}+(\dot{h}-2 )t^{-1}+\dot{c}t^{-2})v)\otimes x^{i+2}$. This indicates the formulae for $\dot{z}=0$
and they are checked by a (long but) straightforward computation.

If $\dot{z}\ne 0$, then from \eqref{A-11} we know that
$$\aligned p^2-2ze=&(x^{-1}(2\dot{z}\partial_t+(\partial_x-1)\dot{z}))^2-\\ &-2
x^{-2}\dot{z}\left(t+\dot{z}(2\partial_x-3))\partial_t+2\dot{z}
\partial_t^2+\frac{(\partial_x-1)(\partial_x-2)}2\dot{z}\right)\\
=&-2\dot{z}tx^{-2}\endaligned $$ which is bijective on $A_{
\dot{z},\dot{h}}$. From \eqref{eq55} we obtain that
$$f=(eq^2+hpq+z(h^2+h)/2+\dot{c})(p^2-2ze)^{-1},$$
that is $f(v\otimes x^i)$ equals
{\small
$$
\left(\left(-\frac{1}{2\dot{z}}-\big(\dot{h}-\frac{1}{2}\big)\frac{d}{dt}-t\big(\frac{d}{dt}\big)^2
-\big(\frac{(\dot{h}-1)(\dot{h}-2)}{4}+\frac{\dot{c}}{2\dot{z}}\big)t^{-1}\right)v\right)\otimes x^{i+2}.$$
}\hspace{-2mm}
This indicates the formulae for $\dot{z}\neq 0$
and they are checked by a (long but) straightforward computation.
\end{proof}

Let $\dot{z}\in \C^*$ and $\lambda \in \C\backslash \Z$. For convenience,
we define the simple $\H$-module $G(\dot{z},\lambda)=\C[x,x^{-1}]$
as follows (see \cite{LMZ}):
$$\aligned &q x^i=x^{i+1},\,\,\,\, p x^i=\dot{z}(\lambda+i)x^{i-1},\,\,\,\,z x^i=\dot{z} x^i.\endaligned$$
Note that the $\H$-module $G(\dot{z},\lambda)$ is an $\H$-submodule
of the simple $\fa$-module $Q'(\dot{z}, \dot{h},\lambda)$ for any
$\dot{h}\in\C$.

Now we can formulate the main result of this section.

\begin{thm}\label{thmtwo}
Let $V$ be a simple weight $\fa$-module.
\begin{enumerate}[$($a$)$]
\item\label{thmtwo.1}
If $V$ is a highest weight module over $\fa$, then $H(V)$ is a
highest weight $\S$-module, which has a unique simple quotient.
\item\label{thmtwo.2}
If $V\cong \bar{M}_{\fa}^-(\dot{z},\dot{h})$ with  $\dot{z}\ne 0$, then
\begin{displaymath}
H(V)\cong M_{\H}^-(\dot{z})^{\S}\otimes
M_{\sl_2}(\dot{h}-\frac{1}{2})^{\S}.
\end{displaymath}
The latter module has a unique simple quotient and this simple quotient is isomorphic to
$M_{\H}^-(\dot{z})^{\S}\otimes \bar{M}_{\sl_2}(\dot{h}-\frac{1}{2})^{\S}$.
\item\label{thmtwo.3}
If $V\cong Q(\lambda,\dot{h})$ for some $\dot{h}\in \C$ and $\lambda\in \C^*$, then $H(V)$ is simple.
\item\label{thmtwo.4}
If $V\cong Q'(\dot{z}, \dot{h},\lambda)$ for some $\dot{h}\in
\C,\dot{z}\in \C^*,\lambda \in \C\backslash \Z$, then
\begin{displaymath}
H(V)\cong G(\dot{z},\lambda)^{\S}\otimes
M_{\sl_2}(\lambda+\dot{h}+\frac{1}{2})^{\S}.
\end{displaymath}
The latter module has a unique simple quotient and this simple quotient is isomorphic to
$G(\dot{z},\lambda)^{\S}\otimes \bar{M}_{\sl_2}(\lambda+\dot{h}+\frac{1}{2})^{\S}$.
\item\label{thmtwo.5}
If $V\cong N_{\dot{z},\dot{h}}$, where $N$ is an infinite
dimensional simple $R_{\dot{z}}$-module, then any simple quotient
of $H(V)$ is isomorphic to  $N_{\dot{c}, \dot{z},\dot{h}}$
for some $\dot{c}\in \C$.
\end{enumerate}
\end{thm}

\begin{proof}
Claim~\eqref{thmtwo.1} is clear. To prove claim~\eqref{thmtwo.2}, we note that the module
$V=\bar{M}_{\fa}^-(\dot{z},\dot{h})$  has a simple $\H$-submodule $M_{\H}^-(\dot{z})$. Then, using \eqref{eq57},
we can extend the action of $\H$ on  $M_{\H}^-(\dot{z})$ to a lowest weight  $\fa$-module
$\bar{M}_{\fa}^-(\dot{z},-1/2)$.  Let $\C v$ be the $1$-dimensional $\fa$-module given by
$ev=pv=qv=zv=0$ and $hv=(\dot{h}-\frac{1}{2})v$. Then we have
$V\cong \C v\otimes \bar{M}_{\fa}^-(\dot{z},-1/2)$. From \cite[Theorem~3]{LMZ} we know that $H(V)\cong
M_{\H}^-(\dot{z})^{\S}\otimes \bar{M}_{\sl_2}(\dot{h}-\frac{1}{2})^{\S}$. The rest of
claim~\eqref{thmtwo.2} now follows easily.

Let $M$ be a nonzero submodule of $H(V)$.  Then $M$ is a weight module. Choose
$0\ne v_n=\sum_{i=0}^k c_i f^i \otimes x^{n-2i} \in M$ such that $k$ is minimal.
If $k>0$, then $0\ne p v_n=-\sum_{i=1}^k ic_i f^{i-1} \otimes x^{n-2i-1}\in M$ which contradicts
our choice of $v_n$. Hence $k=0$ and $v_n\in 1\otimes V$. Now $M$ has to coincide with $H(V)$ since
$V$ is a simple  $\fa$-module. Claim~\eqref{thmtwo.3} follows.

To prove claim~\eqref{thmtwo.4}, we note that $V=Q'(\dot{z}, \dot{h},\lambda)$ has a simple
$\H$-sub\-module $G(\dot{z},\lambda)$. Then we apply \eqref{eq57} to make $G(\dot{z},\lambda)$  into an $\fa$-module
$G(\dot{z},\lambda)^{\fa}$ with $h(x^i)=-(\lambda+i+1/2)x^i$.  Let $\C v$ be the one-dimensional $\fa$-module given by
$ev=pv=qv=zv=0$ and $hv=(\dot{h}+\lambda+\frac{1}{2})v$. Then $V\cong \C v\otimes G(\dot{z},\lambda)_{\fa}$. From
\cite[Theorem~3]{LMZ} we know that $H(V)\cong C(\dot{z},\lambda)^{\S}\otimes
M_{\sl_2}(\lambda+\dot{h}+\frac{1}{2})^{\S}$.  The rest of  claim~\eqref{thmtwo.4} now follows easily.

Finally, we prove claim~\eqref{thmtwo.5}. During the proof of Lemma~\ref{lem15} we computed that in
the module $N_{\dot{z},\dot{h}}$ we have $p^2 (v\otimes x^i)=(t^2v)\otimes x^{i-2}$ if $\dot{z}=0$ and
$(p^2-2ez)(v\otimes x^i)=(-2\dot{z} tv)\otimes x^{i-2}$ if $\dot{z}\ne 0$.
Recall that $t$ acts bijectively on $N$ since $N$ is infinite dimensional. Thus $p^2-2ez$, which is the coefficient at
$f$ in $c$, acts bijectively on $N_{\dot{z},\dot{h}}$ in this case. Let $M$ be any simple quotient of $H(V)$. Then
$c$ acts as the scalar $\dot{c}$ on $M$. Now we have $f (1\otimes V)\subset 1\otimes V$ in $M$, and the action of $f$ is
uniquely determined by $\dot{c}$. Thus we have $M\cong N_{\dot{c},\dot{z},\dot{h}}$. This completes the proof.
\end{proof}

Let $\fa^{op}$ be the parabolic subalgebra of $\mathcal{S}$ spanned by $\{f,h,p,q,z\}$, which is isomorphic to
$\fa$. There are actually many simple weight $\S$-modules that do not contain any simple $\fa$-submodule or
any simple $\fa^{op}$-submodule. For example, if we take $V$ to be a simple weight $\sl_2$-module that is not
highest or lowest weight module and $\dot{c}\ne0$, then the simple $\S$-module $V^{\S}\otimes M_{\H}(\dot{z})^{\S}$
contains neither simple $\fa$-submodules nor simple $\fa^{op}$-submodules.

Taking into account the results of this paper it is natural to ask whether one can classify all
simple weight $\S$-modules or all simple $\fa$-module.
\vspace{5mm}

\noindent {\bf Acknowledgments.} The research presented in this paper was carried out during the visit of the
first author to University of Waterloo and of the second author to Wilfrid Laurier University.  The first author
thanks professors Wentang Kuo and Kaiming Zhao for sponsoring his visit, and University of Waterloo
for providing excellent working conditions. The second author thanks Wilfrid Laurier University for hospitality.
We thank the referee for helpful comments and nice suggestions.

R.L. is partially supported by NSF of China (Grant 11371134) and Jiangsu Government Scholarship
for Overseas Studies (JS-2013-313). \\
V.M. is partially supported by the Swedish Research Council.\\
K.Z. is partially supported by  NSF of China (Grant 11271109) and NSERC.

\vspace{0.2cm}

\noindent \noindent R.L.: Department of Mathematics, Soochow
University,  Suzhou,  P. R. China; e-mail:
{rencail\symbol{64}amss.ac.cn}

\noindent V.M.: Department of Mathematics, Uppsala University, Box
480, SE-751 06, Uppsala, Sweden; e-mail:
{mazor\symbol{64}math.uu.se}

\noindent K.Z.: Department of Mathematics, Wilfrid Laurier
University, Waterloo, Ontario, N2L 3C5, Canada; and College of
Mathematics and Information Science, Hebei Normal (Teachers)
University, Shijiazhuang 050016, Hebei, P. R. China; e-mail:
{kzhao\symbol{64}wlu.ca}


\begin{thebibliography}{999999}
\bibitem[AI]{AI} N. Aizawa, P. S. Isaac. On irreducible representations of the exotic conformal
Galilei algebra. J. Phys. A: Math. Theor. {\bf 44} (2011), 035401.
\bibitem[AIK]{AIK} N. Aizawa, P. S. Isaac, Y. Kimura. Highest weight representations and
Kac determinants for a class of conformal Galilei algebras with central extension. International
Journal of Mathematics {\bf 23} (2012), No. 11, 1--25.
\bibitem[BTLC]{BTLC} J. L. Bueso, J. Gomez Torrecillas, F. J. Lobillo, F. J. Castro, An introduction to effective calculus in quantum groups, in:
Caenepeel, Stefaan et al. (Eds.), Rings, Hopf algebras, and Brauer groups. Proceedings of the Fourth Week on Algebra and
Algebraic Geometry, SAGA-4, Antwerp and Brussels, Belgium, September 12-17, 1996, Marcel Dekker, New York, NY, Lect.
Notes Pure Appl. Math. vol. 197 (1998),  55-83.
\bibitem[B]{B} R.~Block. The irreducible representations of the Lie algebra $\mathfrak{sl}(2)$ and
of the Weyl algebra. Adv. in Math. {\bf 139} (1981), no. 1, 69--110.
\bibitem[BR]{BR} A.J. Bray, A.D. Rutenberg. Growth laws for phase-ordering. Phys. Rev. E {\bf 49}, R27, (1994).
\bibitem[CZ]{CZ}
L.~Cagliero, F.~Szechtman. The classification of uniserial $\sl(2)\rtimes V(m)$-modules and a new
interpretation Racah-Wigner $6j$-symbols. J. Algebra  {\bf 386}  (2013), 142--175.
\bibitem[Di]{Di} J.~Dixmier. Enveloping algebras. Graduate Studies
in Mathematics, {\bf 11}. American Mathematical Society, Providence, RI, 1996.
\bibitem [DDM] {DDM}
V.K.~Dobrev, H.-D.~Doebner and C.~Mrugalla. A q-Schr\"odinger
algebra, its lowest weight representations and generalized
q-deformed heat/Schr\"odinger equations. J. Phys. A {\bf 29} (1996), 5909--5918.
\bibitem [D] {D} B. Dubsky. Classification of simple weight modules with finite-dimensional weight
spaces over the Schr\"odinger algebra, Lin. Algebra Appl. {\bf 443} (2014), 204--214.
\bibitem[DLMZ] {DLMZ} B. Dubsky, R. L\"u, V.~Mazorchuk, K. Zhao. Category $\mathcal{O}$ for the
Schr\"odinger algebra, arXiv: 1402.0334.  Linear Algebra Appl.,  460 (2014), 17-50.
\bibitem[GIl]{GI1} A. Galajinsky,   I. Masterov.   Dynamical realization of l -conformal
Galilei algebra and oscillators. Nuclear Phys. B  {\bf 866}  (2013), no. 2, 212--227.
\bibitem[GI2]{GI2} A. Galajinsky,   I. Masterov.   Remarks on l -conformal extension of
the Newton-Hooke algebra. Phys. Lett. B  {\bf 702}  (2011),  no. 4, 265--267.
\bibitem[HP]{HP} P. Havas and J. Plebanski. Conformal extensions of the Galilei
group and their relation to the Schr\"odinger group. J. Math. Phys. {\bf 19 } (1978), 482--488.
\bibitem[H1]{H1} M.~Henkel. Schrodinger-invariance in strongly anisotropic critical
systems. J. Stat. Phys. {\bf 75}, 1023 (1994).
\bibitem[H2]{H2}
M.~Henkel. Local scale invariance and strongly anisotropic equilibrium
critical systems. Phys. Rev. Lett. {\bf 78}, 1940 (1997).
\bibitem[H3]{H3} M. Henkel. Phenomenology of local
scale-invariance: from conformal invariance to dynamical scaling.
Nucl. Phys. B {\bf 641}, 405 (2002).
\bibitem[H4]{H4}
M. Henkel. Causality from dynamical symmetry: an example from local
scale-invariance. Proceedings of the 7th AGMP conference 24-26 Oct.
2011 at Mulhouse (France), in press (2013), arXiv:1205.5901
\bibitem[HEP]{HEP}
M. Henkel, T. Enss, M. Pleimling. On the identification of
quasiprimary operators in local scale-invariance. J. Phys. A Math.
Gen. {\bf 39}, L589 (2006).
\bibitem[HU]{HU}
M. Henkel, J. Unterberger. Schr\"odinger invariance and space-time symmetries.
Nucl. Phys. B {\bf 660}, 407 (2003).
\bibitem[HS1]{HS1} M. Henkel,  S. Stoimenov. Physical ageing and new representations of some Lie algebras of
local scale-invariance. arXiv:1410.6086.
\bibitem[HS2]{HS2} M. Henkel,  S. Stoimenov. On non-local representations of the ageing algebra,
Nuclear Physics B {\bf 847} (2011), 612--627.
\bibitem[HSSU]{HSSU}  M. Henkel, B. Schott, S. Stoimenov, J. Unterberger. The Poincare algebra in
the context of ageing systems: Lie structure, representations, Apell
systems and coherent states. Confluentes Mathematici {\bf 4}, No. 4 (2012), 1250006, 23pp.
\bibitem[Hu]{Hu} J. E. Humphreys. Introduction to Lie Algebras and Representation Theory, Graduate Texts
in Mathematics {\bf  9}, Springer-Verlag, New York-Berlin, 1978.
\bibitem[Ja]{Ja} N. Jacobson.  Basic algebra. II. W. H. Freeman and Co., San Francisco, Calif., 1980.
\bibitem[LMZ]{LMZ} R. L\"u, V.~Mazorchuk, K. Zhao. On simple modules over conformal Galilei algebras,
J. Pure Appl. Algebra {\bf 218} (2014), 1885-1899.
\bibitem[LZ1]{LZ1} R. L\"u, K. Zhao. Generalized oscillator representations of  the twisted
Heisenberg-Virasoro algebra. Preprint arXiv:1308.6023v1.
\bibitem[M]{M}
O.~Mathieu. Classification of irreducible weight modules. Ann. Inst.
Fourier {\bf  50} (2000), 537--592.
\bibitem[NOR]{NOR}
J. Negro, M. A. del Olmo, A. Rodriguez-Marco. Nonrelativistic conformal groups. J. Math. Phys. {\bf 38} (1997), 3786--3809.
\bibitem[PH]{PH}
A. Picone, M. Henkel. Local scale-invariance and ageing in noisy systems. Nucl. Phys. B {\bf 688} [FS], 217 (2004).
\bibitem [WZ] {WZ}
Y. Wu,  L. Zhu. Simple weight modules for Schroinger algebra. Linear Algebra Appl.
{\bf 438}  (2013), no. 1, 559--563.
\end{thebibliography}
\end{document}